\documentclass{amsart}
\usepackage{color}
\newtheorem{theorem}{Theorem}[section]
\newtheorem{corollary}[theorem]{Corollary}
\newtheorem{lemma}[theorem]{Lemma}
\newtheorem{proposition}[theorem]{Proposition}
\theoremstyle{definition}
\newtheorem{definition}[theorem]{Definition}
\newtheorem{example}[theorem]{Example}

\theoremstyle{remark}
\newtheorem{remark}[theorem]{Remark}
\numberwithin{equation}{section}

\usepackage[all,cmtip]{xy}

\makeatletter 
\renewcommand{\section}{\@startsection{section}{1}{0mm}
 {-\baselineskip}{0.5\baselineskip}{\bf\leftline}}
\makeatother

\begin{document}

\title{Slice Dirac operator over octonions}

\author[M. Jin]{Ming Jin}
\author[G. B. Ren]{Guangbin Ren}
\author[I. Sabadini]{Irene Sabadini}
\address{Ming Jin, Department of Mathematics, University of Science and
Technology of China, Hefei 230026, China}
\email{jinm54$\symbol{64}$mail.ustc.edu.cn}
\address{Guangbin Ren, Department of Mathematics, University of Science and
Technology of China, Hefei 230026, China}
\email{rengb$\symbol{64}$ustc.edu.cn}

\address{Irene Sabadini,
Dipartimento di Matematica, Politecnico di Milano
Via Bonardi, 9, 20133 Milano, Italy}
\email{irene.sabadini@polimi.it}

\thanks{This work was supported by the NNSF  of China (11771412).}
\keywords{Dirac operators, octonions, quaternions,  stem functions, slice regular functions}
\subjclass[2010]{30G35, 32A30, 17A35}

\begin{abstract}
The slice Dirac operator over octonions is a slice counterpart of the Dirac operator over quaternions. It involves a new theory of stem functions, which is the extension from the commutative $ O(1) $ case to the non-commutative $   O(3) $ case. For functions in the kernel of the slice Dirac operator over octonions, we establish the representation formula, the Cauchy integral formula (and, more in general, the Cauchy-Pompeiu formula), and the Taylor  as well as the Laurent series expansion formulas.
\end{abstract}
\maketitle

\section{Introduction}
The purpose of this article is to initiate the study of the slice Dirac operator over octonions. The Dirac operator for quaternions
 \begin{equation}\label{Dirac1}
 D=\frac{\partial}{\partial x_0}+i \frac{\partial}{\partial x_1}+j \frac{\partial}{\partial x_2}+k \frac{\partial}{\partial x_3},
 \end{equation}
has its root in mathematical physics, quantum mechanics, special relativity,  and engineering (see  \cite{A_1, Adler_1, Girard_1}) and it plays a key role in the Atiyah-Singer index theorem (ref. \cite{Bismut_1}). It may be called Dirac operator since it factorizes the $4$-dimensional Laplacian. However, we note that in the literature \eqref{Dirac1} is often called generalized Cauchy-Riemann operator or Cauchy-Fueter operator, see e.g. \cite{Brackx_1, Gurlebeck_1, Sudbery_1}, even though it has been originally introduced in a paper by Moisil, see \cite{Moisil}.

Based on the Dirac operator for quaternions in \eqref{Dirac1}, we shall introduce what we call the slice Dirac operator over octonions, using the slice technique.
This technique was used by Gentili and Struppa for quaternions in \cite{Gentili_1, Gentili_2} and for octonions in \cite{Gentili_3} based on Cullen's approach \cite{Cullen}.
This technique makes it possible to extend some properties of holomorphic functions in one complex variable to the high dimensional and non-commutative case of quaternions.
It has found significant applications especially in operator theory \cite{Alpay_1, CSS_book, Colombo_5}, differential geometry \cite{Gentili_4}, geometric function theory \cite{Ren_1, Ren_2} and it can be generalized to other higher dimensional settings like Clifford algebras \cite{Colombo_1, Colombo_2} and real alternative algebras \cite{Ghiloni_1, Ghiloni_2, Ghiloni_5, Ren_3}.

The heart of the slice technique comes from the slice structure of quaternions $\mathbb H$, namely the fact that $\mathbb H$ can be expressed as union of complex half planes as
$$ \mathbb H=\bigcup_{I \in \mathbb S}\mathbb C_I^+,$$
where $ \mathbb S $ denotes the set of imaginary unit in $ \mathbb H $, and $ \mathbb C_I^+ $ is the upper half plane
$$ \{ x+yI:x\in \mathbb R, y\ge 0 \} .$$
From this decomposition, it is then natural to say that quaternions have a \textit{book structure} since
  $ \mathbb C_I^+ $ plays the role of a page in a  book for any $ I \in \mathbb S $. The real axis $ \mathbb R $ plays the role of the edge of the book in which all the pages of the book intersect, i.e.,
$$\mathbb C_I^+ \cap \mathbb C_J^+=\mathbb R$$
for any $I\neq J$.

The book structure for quaternions  plays the same role as  the sheaf   or fiber  bundle  structure in differential geometry.

It is remarkable  that the topology in the book structure   is no longer the Euclidean one.
Indeed,   the distance compatible with the topology is given by the Euclidean one in a plane, otherwise the distance between any two points from distinct half planes is measured through the path of light via the real axis.

Following the Fueter's construction \cite{Fueter_1}, see also \cite{Sce}, when one considers an open set $O$ in the upper half complex plane $\mathbb C^+$ minus the real line and a holomorphic function $f(x+\iota y)=F_1(x,y)+\iota F_2(x,y)$,  one may define a function defined over the quaternions using the book structure. In fact, if we consider  $q=x+Iy$, $y>0$, for some suitable $I$, we may set $f(q)=f(x+ I y)=F_1(x,y)+I F_2(x,y)$. Note that $\bar q=x-Iy$, $y>0$ and so, by definition,
$f(\bar q)=f(x -I y)=F_1(x,y)-I F_2(x,y)$. Note also that the pair $(F_1,F_2)$ satisfies the Cauchy-Riemann system and thus $f(x+Iy)$  is in the kernel of the Cauchy-Riemann operator $\partial_x+I\partial_y$. If one is willing to extend the definition to the points of the real line, there is a problem since if $q\in\mathbb R$  then $q=x+I0$ and the imaginary unit $I$ is no longer unique.

 To solve this problem, one may consider a weaker notion of book structure and observe that
$$ \mathbb H=\bigcup_{I \in \mathbb S}\mathbb C_I,$$
in other words, we may consider $\mathbb H$ as the union of complex planes.

Following a slight modification of the Fueter's construction, see \cite{Qian},
   let $O$ an open set symmetric with respect to the real axis (possibly intersecting the real axis) and a holomorphic function $f(x+\iota y)=F_1(x,y)+\iota F_2(x,y)$. If $F_1, F_2$ are an even-odd pair in the second variable, namely if they satisfy
\begin{equation}\label{eq:odd-even}\begin{cases}
F_1(x+iy)=F_1(x-iy)
\\
F_2(x+iy)=-F_2(x-iy)
\end{cases} \forall\ x,y\in \mathbb R,
\end{equation}
we may define a function on an open set in $\mathbb H$ (suitably constructed using $ O$)
Note that these conditions immediately imply that
$$
f(x+Iy)=f(x+(-I)(-y))
$$
so that $f$ is well defined. Moreover, the fact that $F_2$ is odd in the second variable imply that $F_2(x,0)=0$, thus $f$ is well defined also at real points.
This second approach is the one that we will generalize to the octonionic case.

To this end, we set
$$F\equiv\left(
\begin{array}{c}
F_1\\
F_2\\
\end{array}\right), \qquad z=x+iy\equiv\left(
\begin{array}{c}
x\\
y\\
\end{array}\right),
$$
and we consider
\begin{equation*}
g=\left(
\begin{array}{cc}
1&0\\
0&-1\\
\end{array}\right)\in O(1)
\end{equation*}
where $O(1)$ is identified with the group of matrices $\left\{\left(
\begin{array}{cc}
1&0\\
0&1\\
\end{array}\right), \left(
\begin{array}{cc}
1&0\\
0&-1\\
\end{array}\right)\right\}$.
Then
we have
\begin{equation*}
gz=\left(
\begin{array}{c}
x\\
-y\\
\end{array}\right), \qquad
 g^{-1}=\left(
\begin{array}{cc}
1&0\\
0&-1\\
\end{array}\right),
\end{equation*}
\\
so (\ref{eq:odd-even}) can be rewritten as
\begin{equation}
 \left(\begin{array}{c}
F_1(z)\\
F_2(z)\\
\end{array}\right)
=
g^{-1} \left(\begin{array}{c}
F_1(gz)\\
F_2(gz)\\
\end{array}\right).
\end{equation}
Thus, following \cite{Rin}, we impose that
$$ F(z)=g^{-1}F(gz),\quad  \forall\ g \in O(1), $$
and any $F$ satisfying this condition is called {\em stem function}.

We can regard at this construction as the commutative stem function theory since $ F $ is invariant under the commutative group $ O(1) $.

As we shall see, the significant property of the slice regular function in $\mathbb H$ (non-commutative counterpart of holomorphic functions, i.e. holomorphic maps depending on the parameter $I\in\mathbb S$) is given by
 the representation formula, which demonstrates  that   any  slice regular function  is completely  determined by its evaluation at any two distinct half planes, or {\em pages} in this description.

In order to extend the slice theory for the Cauchy-Riemann operator over quaternions into the slice theory for the slice Dirac operator over octonions, we need to introduce a modified  theory of stem functions. It turns out that  the corresponding notion of  stem function is invariant under the non-commutative group $   O(3) $.
It will result in a new  form of  the representation formula, expressed in term of a quaternionic matrix.

We point out that  the  non-commutative and non-associative setting of octonions  has found significant applications in  the universal model of $M$-theory, in which the universe is given by the product of the $4$-dimensional Minkowski space with a $G_2$-manifold of very small scale.
Here the exceptional Lie group $G_2$ is an automorphism group of octonions (ref. \cite{Baez_1, Grigolian_1}).

  We conclude this introduction with a remark about our definition of intrinsic and stem functions.  Rinehart \cite{Rin}  studied the intrinsic functions as self-mappings of an associative algebra. In contrast, our  intrinsic functions have distinct dimensions for their definition and  target domains, and are constructed in the non-associative setting;
  see also  \cite{DSce, Sce}. Fueter \cite{Fueter_1} initiated the study of stem functions  for complex-valued functions in his construction of  radially holomorphic functions on the space of quaternions; see \cite{Ghiloni_1} for its recent development. However, their considerations are all restricted to the commutative $O(1)$ setting. In this paper we initiate the study in the non-commutative $  O(3)$ setting. It is interesting to note that the procedure we followed in this paper may lead to further generalizations to higher dimensional algebras.

The structure of this paper is as follow:
in section 2, we recall some important properties  of octonionic algebra $ \mathbb O $.
In section 3, we introduce the book structure in the octonionic algebra in terms of quaternionic subspaces and the stem function for the non-commutative group $  O(3)$; we also provide the representation formula which can be written via a quaternionic  matrix.
In section 4 we introduce the slice Dirac operator and a splitting property for slice Dirac functions. Section 5 contains the Cauchy-Pompeiu integral formula for  slice functions  and the Cauchy integral formula for  slice Dirac-regular functions.
Finally, in  section 6 we give the expansion of slice Dirac-regular as Taylor series as well as  Laurent series.

\section{The algebra of octonions}
The algebra of octonions $ \mathbb O $ is a real, alternative, non-commutative and non-associative division algebra (see for example \cite{Schafer_1}).
It is isomorphic to $ \mathbb R^8 $ as a real inner product  vector space and it can be equipped with the standard orthogonal basis : $ e_0=1,e_1,\dots,e_7 $.

The multiplication between elements in the basis $ e_0,e_1,\dots,e_7 $ is defined by
$$ e_ie_j=-\delta_{ij}+\varepsilon_{ijk}e_k, \ \ \ \forall\  i, j, k \in \{ 1,2,\dots,7 \}. $$
Here $ \delta_{ij} $ is the Kronecker symbol and
\[ \varepsilon_{ijk}=
\begin{cases} (-1)^{\sigma(\pi)}  &\textrm{if}\ \   (i, j, k) \in \pi(\Sigma),
\\
0 & \textrm{otherwise}.
\end{cases} \]
where  $ \pi $ is a permutation, $\sigma(\pi)$ its sign, and
 $$ \Sigma=\{ (1, 2, 3), (1, 4, 5), (2, 4, 6), (3, 4, 7), (5, 3, 6), (6, 1,7), (7, 2, 5) \}. $$

The octonionic algebra $ \mathbb O $ also can be generated  from the quaternions algebra $ \mathbb H $ by the famous Cayley-Dickson process.
Let $ \{e_0=1, e_1, e_2, e_3:=e_1e_2 \} $ be a basis of $ \mathbb H $. Then
every $ x \in \mathbb O $ can be expressed as $ x=a+e_4b $, where $ a, b \in \mathbb H $, and $ e_4 $ is a fixed imaginary unit in $ \mathbb O $ not belonging to $\mathbb H=\{e_0,e_1,e_2,e_3\}$.
The addition and multiplication are defined as follow: for any $ x=a+e_4b, \  y=c+e_4d \in \mathbb O$,
\[ x+y:=(a+c)+e_4(b+d), \]
\[ xy:=ac-d\overline b+e_4(\overline a d + b c ).   \]
This means the following relations hold:
\begin{lemma}\label{ee}
Let $ a,b,e_4 $ as above, then
\begin{enumerate}
\item $a(e_4b)=e_4(\overline a b)$,
\item $(e_4a)b=e_4(ba)$,
\item $(e_4a)(e_4b)=-b\overline a$.
\end{enumerate}
\end{lemma}

The two definitions above of the octonionic algebra $ \mathbb O $ are  equivalent  by setting $$ e_5:=e_4e_1,\quad  e_6:=e_4e_2, \quad e_7:=e_4e_3. $$

Every $ x \in \mathbb O $ can be written as
$$ x=x_0+\sum_{k=1}^7e_k x_k, \qquad\forall\
 x_k\in\mathbb R. $$
We can introduce its conjugate $$ \overline x:=x_0-\sum_{k=1}^7e_k x_k,$$
and then set $$|x|^2:=x \overline x=\sum_{k=0}^7x_k^2. $$
The modulus is multiplicative, i.e.
$$ |xy|=|x||y|, \qquad \forall\ x,\ y \in \mathbb O.$$

In the sequel, given $ x \in \mathbb O $ we introduce a left multiplication operator $ L_x: \mathbb O \to \mathbb O $, defined as
$$ L_x z= x z, \quad  \forall\  z \in \mathbb O. $$ In general, for any $ x,\ y \in \mathbb O, \ L_xL_y \ne L_{xy} $, but equality may hold when suitable assumptions hold:

\begin{theorem}(Artin theorem, \cite{Schafer_1})\label{Artin}.
The subalgebra generated by any two elements of an alternative algebra is associative.
In particular, for all $ r \in \mathbb R $, and for all $  x \in \mathbb O $.
\begin{enumerate}
\item $ L_xL_{\overline x}=L_{x\overline x} $,
\item $ L_rL_x =L_{rx} $.
\end{enumerate}
\end{theorem}
It is also useful to recall the following well known result:
\begin{theorem}(Moufang Laws, \cite{Schafer_1}).
Let $ x, y, z \in \mathbb O $, then
\begin{enumerate}
\item $L_{xyx}z=L_xL_yL_x z$,
\item $L_zL_{xy}x=L_{zx}L_yx$,
\item $L_xL_{yz}x=L_{xy}L_z x$.
\end{enumerate}
\end{theorem}

\section{Stem function in the octonionic setting}
Let $ \mathbb O $ be the algebra of octonions.  The set of its imaginary units is a sphere of dimension six
$$\mathbb S^6 := \{ x \in \mathbb O:x^2=-1 \}.$$
Let
$$  \mathbb I:=(1,I,J,K)\in \mathbb O^4,$$
with
the triple $I,J, K$ satisfying
$$I, J \in \mathbb S^6, \qquad  I \perp J, \qquad  K=IJ. $$
The set of all such  row vectors   $\mathbb I$ is denoted by $ \mathcal N$.
For any $  \mathbb I:=(1,I,J,K)\in \mathcal N,$
we  consider  the algebra of quaternions generated by it, i.e.,
$$\mathbb H_{\mathbb I}=\textrm{span}_{\mathbb R}\{1,I,J,K\}.$$
We can endow  the octonionic algebra with  a structure that we still call \textit{book} structure
$$ \mathbb O=\bigcup_{\mathbb I \in \mathcal N} \mathbb H_{\mathbb I}, $$
as we prove in the following result:

\begin{proposition}\label{bookOct} The octonionic algebra has the structure:
$$ \mathbb O=\bigcup_{\mathbb I \in \mathcal N} \mathbb H_{\mathbb I}. $$
\end{proposition}
\begin{proof} Any  $ x \in \mathbb O $ can be written as the sum of its real part $x_0$ and its imaginary part ${\rm Im}(x)=\sum_{k=1}^7e_kx_k$. Therefore,  it  can be further expressed   as
$$ x=x_0+Iy $$
with $ x_0,\ y \in \mathbb R $ and $I={\rm Im}(x)/|{\rm Im}(x)|$. We have that
$$
I^2=\frac{1}{|{\rm Im}(x)|^2}(\sum_{k=1}^7e_kx_k)(\sum_{k=1}^7e_kx_k)=-\frac{1}{|{\rm Im}(x)|^2}\sum_{k=1}^7 x_k^2=-1,
$$
thus $I \in \mathbb S^6 $. Now we can choose $ J,\  K \in \mathbb S^6 $
such that $ \mathbb I:=(1,\ I,\ J,\ K)\in\mathcal N $. Hence $ x_0+ I y\in \mathbb H_{\mathbb I} $.
\end{proof}
We note that, in general, any $x\in\mathbb O$ belongs to more than one quaternionic space, as the following
example shows.
\begin{example}  Let $\{1, e_1, \cdots, e_7\}$ be a standard basis of $\mathbb O$ and consider $$x=1+e_1+e_2+e_3+e_4 .$$ According to Proposition  \ref{bookOct}, we have $x=1+I\sqrt 4$ where $I=\frac{1}{\sqrt 4}(e_1+e_2+e_3+e_4)$, so $x\in\mathbb H_{\mathbb I}$, $\mathbb I=(1,\ I,\ J,\ K)$ where $J$, $K$ are any two elements orthogonal to $I$ such that $\mathbb I\in\mathcal N$.
Take now $$ I'=e_1,\  J'=\frac{e_2+e_3+e_4}{\sqrt{3}} ,\ K'=I'J'=e_1\frac{e_2+e_3+e_4}{\sqrt{3}}.$$
It is easy to see that $$I',\ J', \ K' \in \mathbb S^6, \qquad {I'}^2={J'}^2={K'}^2=-1, \qquad J'K'=-K'J',\quad  K'I'=-I'K'. $$
Since $x=1+I'+\sqrt 3 J'$, we have   $ x \in \mathbb H_{\mathbb I'}$, where
   $ \mathbb I:=(1,\ I',\ J',\ K') $.
\end{example}

 Let $ O(4) $ be the group of orthogonal transformations of $ \mathbb R^4 $, and let $   O(3) $ be its subgroup keeping
the real axis  invariant.
{
Therefore,  any $ g\in   O(3) $ can be regarded as matrix in the form
\begin{equation*}
\begin{aligned}
g={\left( \begin{array}{cc}
1&0\\
0&P\\
\end{array} \right)}.
\end{aligned}
\end{equation*}
where  $ P$ is an orthogonal transformations of $ \mathbb R^3 $. The transformation $g: \mathbb R^4\longrightarrow \mathbb R^4$  can be naturally extended to be a map
$g: \mathbb O^4\longrightarrow \mathbb O^4$ via
\begin{equation*}
\begin{aligned}
ga=
g{\left( \begin{array}{c}
a_0\\
a_1\\
a_2\\
a_3\end{array} \right)}.
\end{aligned}
\end{equation*}
 for any $ a=(a_0,a_1,a_2,a_3) \in \mathbb O^4 $.
}

\begin{definition}
Let  $ [\Omega] $  be an open subset  of $ \mathbb R^4 $. If $ F :[\Omega] \to \mathbb O^4 $ is a
{\em $  O(3)-$intrinsic function}, i.e. for any $ x  \in [\Omega] $ and for any  $ g \in   O(3) $ such that $ g x \in [\Omega] $, it satisfies
\begin{equation} \label{eq:variant-1}
F(x)= g^{-1}F(g x),
\end{equation}
then F is called an {\em $\mathbb O-$stem function} on $ [\Omega] $.
\end{definition}
\begin{remark}
{\rm We point out that it is not reductive to assume that $[\Omega]\subseteq\mathbb R^4$ is $  O(3)$-intrinsic, otherwise in the previous definition we should consider the subset $[\Omega]'$ of $[\Omega]$ such that $x\in[\Omega]'$ if and only if $ g x \in [\Omega]'$ for any  $ g \in   O(3)$. But this is equivalent to assume that $[\Omega]'$ is $  O(3)$-intrinsic.\\
 We also recall that in the  quaternionic case, the stem function is complex intrinsic i.e. it is invariant under the commutative group $O(1)$. In other words, $f(\bar z)=\overline{f(z)}$ where $\bar z$ denotes the complex conjugation.
In our setting, the stem function is intrinsic under the non-commutative group $   O(3)$ and this set is evidently non-empty since it contains,
e.g., $F_i(x)= x_i c, \ i=1,2,3,4 $ where $ c $ is a constant in $\mathbb O $.}
\end{remark}

With the book structure, we can define a slice function by lifting a stem function.
In fact, if a $ \mathbb O^4-$valued function $F$ defined on $ \mathbb R^4 $ is an $  O(3)$-intrinsic function,
then there exists a slice function $ f:\mathbb O \to \mathbb O $ such that the following diagram commutes for all $ \mathbb I=(1, I, J, K) \in \mathcal N $:
\[
\xymatrix{
\mathbb O \ar[r]^f &\mathbb O \\
 \mathbb R^4 \ar [u]^{\phi_{\mathbb I}} \ar [r]^F &\mathbb O^4 \ar [u]^{\widetilde{\phi}_{\mathbb I}} }
\]
where
$$ \phi_{\mathbb I}(x)=x_0+Ix_1+Jx_2+Kx_3=:\mathbb I x^T, \quad \forall\ x=(x_0,x_1,x_2,x_3) \in \mathbb R^4, $$
and
$$ \widetilde{\phi}_{\mathbb I}(y)=\mathbb I y^T, \quad  \forall\  y \in \mathbb O^4.  $$
Here we denote  by $x^T$
  the transpose of the row vector $x=(x_0,x_1,x_2,x_3)$ and similarly for $y^T$.

Given an open subset $ [\Omega] $ of $ \mathbb R^4 $, we consider the axially symmetric open set in $ \mathbb O $ generated by $[\Omega]$, defined as
$$ \Omega:=\Big\{q=\mathbb I x^T \in \mathbb O : \mathbb I  \in \mathcal N, x  \in [\Omega] \Big\}. $$
If $ [\Omega] $ is a domain, then $ \Omega $ is an axially  symmetric domain.

For any $ x=(x_0,x_1,x_2,x_3) \in \mathbb R^4 $, we consider the three involutions
$$ \alpha(x)=(x_0,x_1,-x_2,-x_3), $$
$$ \beta(x)=(x_0,-x_1,x_2,-x_3), $$
$$ \gamma(x)=(x_0,-x_1,-x_2,x_3). $$
{
Let $\mathbb I\in\mathcal N$ be fixed arbitrarily. By virtue of the identification of $\mathbb H_{\mathbb I}$ with $\mathbb R^4$,  the map $\alpha$ can be identified with the map
$$\alpha_{\mathbb I}: \mathbb H_{\mathbb I}\longrightarrow \mathbb O$$
defined by
$$\alpha_{\mathbb I}(\mathbb I x^T):=\mathbb I \alpha(x)^T.$$
To ease the notation we still write $\alpha$ instead of
 $ \alpha_{\mathbb I} $.
The same convention is adopted in the sequel for
 $  \beta,\gamma, \mathcal F, V, V_{\alpha}, \mathbb V, \mathcal V, P_{\alpha},\mathbb P_{\alpha}, \mathcal P_{\alpha} , A_{\alpha}$,   and $B_{\alpha}$ which also depend on $ \mathbb I \in \mathcal N $.
}

\begin{definition}
For any open subset $ [\Omega] $ of $ \mathbb R^4 $,
we define the symmetrized set $ \widetilde \Omega$ as
$$ \widetilde \Omega:=\{ \mathbb I x^T \in \mathbb O: \mathbb I \in \mathcal N, x \in [\Omega] \  \textrm{such that} \ \alpha(x),\beta(x),\gamma(x) \in [\Omega] \}. $$
\end{definition}
It is easy to check  that $$ \widetilde \Omega \subset \Omega.$$

\begin{definition}\label{slice}
Let $\Omega$ be an open set in $\mathbb O$.
Any $ \mathbb O $-stem function $ F: [\Omega] \to \mathbb O^4 $   induces a function $$ f=\mathcal L(F): \Omega \to \mathbb O $$
 defined by $$ f(q)=\mathbb I F(x)^T$$ for any
 $ q\in \Omega $ with
  $ q=\mathbb I x^T$ for some $\mathbb I\in\mathcal N$. We say that $f$ is a  (left) slice function (induced by $F$).
\end{definition}
Since, in general, any element in $\mathbb O$ may belong to more than one $\mathbb H_{\mathbb I}$ we need to prove the following:
\begin{proposition}
Definition \ref{slice} is well-posed.
\end{proposition}
\begin{proof}
Assume that $ q\in\mathbb O $ can be written in two ways as $$q=\mathbb I x^T=\mathbb I' {x'}^T,$$ we have  to show  that $$f(\mathbb I x^T)=f(\mathbb I^\prime  {x^\prime}^T).$$
We divide the proof into various cases.\\
Case 1: Assume that $\mathbb I\not=\mathbb I'$ but $\mathbb H_{\mathbb I}=\mathbb H_{\mathbb I'}$. Then there exist $g\in   O(3)$ such that $\mathbb I'=\mathbb I g$. This means that  $g{x'}^T=x^T$. Since $F$ is an $\mathbb O$-stem function,  we have
$$
F({x'}^T)=F(g^{-1}x^T)=g^{-1}F(x^T)
$$
so that  $$\mathbb I'F({x'}^T)=\mathbb I'g^{-1}F(x^T)$$
 which yields $\mathbb I'F({x'}^T)=\mathbb IF(x^T)$ and the assertion follows.
\\
Case 2: If $ \mathbb H_{\mathbb I}\ne \mathbb H_{\mathbb I'} $, we claim that  $\mathbb H_{\mathbb I} $ and $ \mathbb H_{\mathbb I'} $ intersect at $ \mathbb C_{I} $ for some
$ I \in \mathbb S^6 $. Indeed, since $q=\mathbb I x^T=\mathbb I' {x'}^T$,  there exists $ y_0,y_1 \in \mathbb R $ and $ I \in \mathbb S^6 $
 such that $ q=y_0+Iy_1\in \mathbb C_I $ and the claim follows. Therefore, we can choose $J_1, J_1^{\prime} \in \mathbb S^6 $ respectively such that
$$  \mathbb H_{\mathbb I}= \mathbb H_{\mathbb I_1},\qquad
 \mathbb H_{\mathbb I'}= \mathbb H_{\mathbb I_1^{\prime}} $$
 where $ \mathbb I_1=(1, I, J_1 , IJ_1) $ and $ \mathbb I_1^{\prime}=(1, I, J_1^{\prime} , IJ_1^{\prime}) $.
 Since $q=  \mathbb I x^T= \mathbb I' {x'}^T \in  \mathbb C_I $, it can be written as
   $$ q=\mathbb Ix^T=\mathbb I_1 y^T=\mathbb I_1^{\prime}y^T=\mathbb I'x'^T$$ for some  $ y=(y_0,y_1,0,0) \in \mathbb R^4 $.
 The computations in Case 1 then shows
$$ \mathbb IF(x)^T=\mathbb I_1 F(y)^T=\mathbb I_1^{\prime}F(y)^T=\mathbb I'F(x')^T.$$
In conclusion, Definition \ref{slice} is well--posed.
\end{proof}

\begin{definition}
Let $\Omega$ be an open set in $\mathbb O$. We set
$$ S(\Omega):=\Big\{ f:\Omega \to \mathbb O\ |\ f=\mathcal L(F),\ F:[\Omega] \to \mathbb O^4 \ {\rm is  \  an} \  \mathbb O-{\rm stem\  function} \Big\}, $$
in other words, $ S(\Omega)$ denotes the collection of slice functions on $ \Omega$.
\end{definition}
 Now we provide  the representation formula of slice functions in terms of a quaternion matrix:

\begin{theorem}\label{rep.s formula}
Let $ f $ be a slice function on an axially  symmetric set $ \Omega $ in $\mathbb O$. Let $q\in\mathbb O$ and let $q=\mathbb I x^T$, for
  $\mathbb I \in\mathcal N$ and $x\in\mathbb R^4$. Then for any $p:=\mathbb I^\prime x^T$ with   $\mathbb I^\prime \in\mathcal N$ the following formula holds:
\begin{equation}
\begin{aligned}
f(p)=
&{\left( \begin{array}{cccc}
1&I'&J'&K'
\end{array} \right )}
\left(\frac{1}{4}{\left( \begin{array}{cccc}
1&1&1&1\\
-I&-I&I&I\\
-J&J&-J&J\\
-K&K&K&-K
\end{array} \right )}
  {\left( \begin{array}{c}
f(q)\\
f(\alpha(q))\\
f(\beta(q))\\
f(\gamma(q))\end{array} \right )}\right).
\end{aligned}
\end{equation}
\end{theorem}

\begin{proof}
Since $ \Omega $ is an axially  symmetric set, we have  $ \alpha(q),\beta(q),\gamma(q) \in \Omega $ for any $ q \in \Omega $. By   definition,
\begin{equation}
{\left( \begin{array}{cccc}
1&1&1&1\\
1&1&-1&-1\\
1&-1&1&-1\\
1&-1&-1&1\end{array} \right )}
{\left( \begin{array}{c}
F_0(x)\\
IF_1(x)\\
JF_2(x)\\
KF_3(x)\end{array} \right )}
=
{\left( \begin{array}{c}
f(q)\\
f(\alpha(q))\\
f(\beta(q))\\
f(\gamma(q))
\end{array} \right )}
\end{equation}
so that
\begin{equation}
\begin{aligned}
{\left( \begin{array}{c}
F_0(x)\\
IF_1(x)\\
JF_2(x)\\
KF_3(x)
\end{array} \right )}
=&
{\left( \begin{array}{cccc}
1&1&1&1\\
1&1&-1&-1\\
1&-1&1&-1\\
1&-1&-1&1
\end{array} \right )}^{-1}
{\left( \begin{array}{c}
f(q)\\
f(\alpha(q))\\
f(\beta(q))\\
f(\gamma(q))\end{array} \right )}
\\
=&\frac{1}{4}{\left( \begin{array}{cccc}
1&1&1&1\\
1&1&-1&-1\\
1&-1&1&-1\\
1&-1&-1&1
\end{array} \right )}
{\left( \begin{array}{c}
f(q)\\
f(\alpha(q))\\
f(\beta(q))\\
f(\gamma(q))\end{array} \right )}
\end{aligned}
\end{equation}
Thanks to Artin Theorem, see Theorem \ref{Artin}, we get
\begin{equation}\label{eq35}
{\left( \begin{array}{c}
F_0(x)\\
F_1(x)\\
F_2(x)\\
F_3(x)\end{array} \right )}
=
\frac{1}{4}{\left( \begin{array}{cccc}
1&1&1&1\\
-I&-I&I&I\\
-J&J&-J&J\\
-K&K&K&-K
\end{array} \right )}
{\left( \begin{array}{c}
f(q)\\
f(\alpha(q))\\
f(\beta(q))\\
f(\gamma(q))\end{array} \right )}.
\end{equation}
By the definition of slice functions,   for any $ \mathbb I'=(1,I',J',K') \in \mathcal N $ we then have
\begin{equation*}
\begin{aligned}
&f(x_0+I'x_1+J'x_2+K'x_3)=F_0(x)+I'F_1(x)+J'F_2(x)+K'F_3(x)
\\
&={\left( \begin{array}{cccc}
1&I'&J'&K'
\end{array} \right )}
\left(\frac{1}{4}{\left( \begin{array}{cccc}
1&1&1&1\\
-I&-I&I&I\\
-J&J&-J&J\\
-K&K&K&-K
\end{array} \right )}
{\left( \begin{array}{c}
f(q)\\
f(\alpha(q))\\
f(\beta(q))\\
f(\gamma(q))\end{array} \right )}\right).
\end{aligned}
\end{equation*}
\end{proof}

\begin{remark}
The representation formula can be briefly expressed as
$$ f(\mathbb I'x^T)=\mathbb I'(M_{\mathbb I}\mathcal F(q)),$$
 where $ \mathbb I'=(1,I',J',K') \in \mathcal N $, $q=\mathbb I x^T$,
$$ \mathcal F(q)=(f(q),f(\alpha(q)),f(\beta(q)),f(\gamma(q)))^T,$$ and
\[M_{\mathbb I}=
\frac{1}{4}{\left( \begin{array}{cccc}
1&1&1&1\\
-I&-I&I&I\\
-J&J&-J&J\\
-K&K&K&-K
\end{array} \right )}.
\]
This representation is very useful to prove further properties of slice functions.
\\
Moreover, notice that $2 M_{\mathbb I} $ is an orthogonal matrix with elements in $ \mathbb H_{\mathbb I } $, i.e. $$ 2 M_{\mathbb I } \in O(\mathbb H_{\mathbb I}).$$
\end{remark}

\noindent
The following result shows that
the slice function $f(\mathbb Ix^T)$ is a linear function of $\mathbb I$.

\begin{theorem}\label{rep-formula}
Let $ f $ be a slice function on an axially  symmetric set $ \Omega $. Then
 the octonionic-valued vector function $ M_{\mathbb I}\mathcal F(q) $  depends only on $ x $ but not on $ \mathbb I, \mathbb I'$ and $f(\mathbb I'x)=\mathbb I' (M_{\mathbb I}\mathcal F(q)) $ is a linear function in $\mathbb I'$.
\end{theorem}

\begin{proof} By construction,  $ M_{\mathbb I}\mathcal F(q) $ is  independent  of $\mathbb I'$.  Theorem \ref{rep.s formula} shows   that
$$ f(\mathbb I'x^T)=\mathbb I'(M_{\mathbb I}\mathcal F(q)) $$
holds for any $\mathbb I$, which implies that $ M_{\mathbb I}\mathcal F(q) $ is  independent  of $\mathbb I$.
(In alternative, one can prove the assertion noting that \eqref{eq35} shows that $$M_{\mathbb I}\mathcal F(q)=\left( \begin{array}{c}
F_0(x)\\
\vdots \\
F_3(x)\end{array} \right )$$ and so $M_{\mathbb I}\mathcal F(q)$ is independent of $\mathbb I$). Moreover, the linearity in $\mathbb I'$ is immediate.
\end{proof}

\begin{remark}
Also the representation formula for quaternionic slice regular functions can be written in matrix form. In fact, for any $ I,J \in \mathbb S $ where $ \mathbb S $ is the set of imaginary unit of quaternions $ \mathbb H $, and for any $ x, y\in \mathbb R $, the representation formula can be written as
\begin{equation}
f(x+Jy)=
\frac{1}{2}{\left( \begin{array}{cc}
1&J
\end{array} \right )}
{\left( \begin{array}{cc}
1&1\\
-I&I
\end{array} \right )}
{\left( \begin{array}{c}
f(x+Iy)\\
f(x-Iy)\end{array} \right )}
\end{equation}
\end{remark}

\section{Slice Dirac operator}
In this section, we introduce the slice Dirac operator in $ \mathbb O $ and establish the corresponding splitting lemma. We begin by recalling the Dirac operator \eqref{Dirac1} introduced in Section 1:
$$ D=\frac{\partial}{\partial x_0}+i\frac{\partial}{\partial x_1}+j\frac{\partial}{\partial x_2}+k\frac{\partial}{\partial x_3}= {\left( \begin{array}{cccc}
1&i & j &k
\end{array} \right )}
{\left( \begin{array}{cc}
\dfrac{\partial}{\partial x_0}\\
\vdots \\
\dfrac{\partial}{\partial x_3}
\end{array} \right )}= {\left( \begin{array}{cccc}
1&i & j &k
\end{array} \right )} \mathsf D
$$
where $\mathsf D=\left( \dfrac{\partial}{\partial x_0}\ \,\dfrac{\partial}{\partial x_1}\ \,\dfrac{\partial}{\partial x_2}\ \,\dfrac{\partial}{\partial x_3}\right)^T$,
and its conjugate operator
$$ \overline D =\frac{\partial}{\partial x_0}-i\frac{\partial}{\partial x_1}-j\frac{\partial}{\partial x_2}-k\frac{\partial}{\partial x_3}= {\left( \begin{array}{cccc}
1&-i & -j &-k
\end{array} \right )} \mathsf D.$$
For any fixed $\mathbb I=(1, I, J, K)\in\mathcal N$,
we define the {\em slice Dirac operator} in $ \mathbb O $ as
 \begin{equation}\label{duebolli}
 D_{\mathbb I}=\frac{\partial}{\partial x_0}+I\frac{\partial}{\partial x_1}+J\frac{\partial}{\partial x_2}+K\frac{\partial}{\partial x_3} ={\left( \begin{array}{cccc}
1& I & J & K
\end{array} \right )} \mathsf D.
 \end{equation}
In the sequel, the restriction $f\left.\right|_{\mathbb H_{\mathbb I}}$ of a function $f$ to $\mathbb H_{\mathbb I}$  shall be denoted by $f_{\mathbb I}$:
$$f_{\mathbb I}=f\left.\right|_{\mathbb H_{\mathbb I}}.$$

We now introduce a main definition:


\begin{definition}\label{slice Dirac}
Let $ \Omega $ be an axially  symmetric domain in $ \mathbb O $ and let  $f \in S(\Omega) \cap C^1(\Omega)$ so that $f=\mathcal L(F)$, $f(q)=\mathbb I F(x)^T$, where $q=\mathbb I x^T$, $F=[F_0,F_1,F_2,F_3]$. If $F$  satisfies
\begin{equation}\label{CFcondition}
\left( \begin{array}{cccc}
\partial_{x_0}&-\partial_{x_1}&-\partial_{x_2}&-\partial_{x_3}\\
\partial_{x_1}&\partial_{x_0}&-\partial_{x_3}&\partial_{x_2}\\
\partial_{x_2}&\partial_{x_3}&\partial_{x_0}&-\partial_{x_1}\\
\partial_{x_3}&-\partial_{x_2}&\partial_{x_1}&\partial_{x_0}
\end{array} \right )
\left( \begin{array}{c} F_0 \\ F_1\\ F_2\\ F_3
\end{array} \right )=\left( \begin{array}{c} 0 \\ 0\\ 0\\ 0
\end{array} \right )
\end{equation}
then $f$ is called a (left) slice Dirac-regular function in $ \Omega $.
\end{definition}

We denote the set of slice Dirac-regular functions on the axially symmetric set $ \Omega $ by
$ SR(\Omega)$.

\begin{proposition}
Let $ \Omega $ be an axially  symmetric domain in $ \mathbb O $ and let  $f \in S(\Omega) \cap C^1(\Omega)$. Then $f$ is (left) slice Dirac-regular if and only if
 \begin{equation}\label{DiracI} D_{\mathbb I}f(q)=0,  \qquad \forall\   q \in \Omega\cap\mathbb H_{\mathbb I}=:\Omega_{\mathbb I}
 \end{equation}
 and for all $\mathbb I\in\mathcal N$.
\end{proposition}
\begin{proof}
Let $f$ be slice Dirac-regular and let $q \in \Omega\cap\mathbb H_{\mathbb I}$, $q=\mathbb Ix^T$. Then using \eqref{CFcondition}  we have:
\begin{equation}\label{comp}
\begin{split}
D_{\mathbb I}f(q)&=D_{\mathbb I} \mathbb I F(x^T)=D_{\mathbb I}(F_0+IF_1+JF_2+KF_3)\\
&=(\partial_{x_0}F_0-\partial_{x_1}F_1-\partial_{x_2}F_2-\partial_{x_3}F_3)+I(
\partial_{x_1}F_0+\partial_{x_0}F_1-\partial_{x_3}F_2+\partial_{x_2}F_3)\\
&+J(\partial_{x_2}F_0+\partial_{x_3}F_1+\partial_{x_0}F_2-\partial_{x_1}F_3)+K(
\partial_{x_3}F_0-\partial_{x_2}F_1+\partial_{x_1}F_2+\partial_{x_0}F_3)\\
&=0.
\end{split}
\end{equation}
Conversely, let us assume that the slice function $f$ is such that \eqref{DiracI} holds for all $\mathbb I\in\mathcal N$. Let us fix an arbitrary $\mathbb I\in\mathcal N$ and $q=\mathbb I x^T$ and let us impose that $f(q)=\mathbb I F(x^T)$ satisfies \eqref{DiracI}. Computations as in \eqref{comp} show that \eqref{CFcondition} holds, by arbitrarity of $\mathbb I$.
\end{proof}

\begin{remark}\label{DDD}
We note that \eqref{DiracI} is well-defined.
Indeed, for any $ q \in \mathbb O $ there exist $ \mathbb I \in \mathcal N $ such that $ q \in \mathbb H_{\mathbb I} $ and $ q=(1,I,J,K)(x_0,x_1,x_2,x_3)^T$.
 It can be also written as
$ q=((1,I,J,K)g^{-1})(y_0,y_1,y_2,y_3)^T$ for any $ g \in   O(3) $ and $ y^T=g x^T $.
By the chain rule, it is direct to show that
$$  (1,I,J,K)(\partial_{x_0}, \partial_{x_1}, \partial_{x_2}, \partial_{x_3})^Tf(q)=  ((1,I,J,K)g^{-1})(\partial_{y_0},\partial_{y_1},\partial_{y_2},\partial_{y_3})^Tf(q) $$
which implies the claim.
\end{remark}

{
\begin{example}\label{eg1}
We consider the function $ F=(F_0,F_1,F_2,F_3):\mathbb R^4 \to \mathbb O^4 $ defined by
\[
\begin{cases}  F_0(x)=3x_0,
\\
 F_1(x)=x_1,
 \\
F_2(x)=x_2,
\\
F_3(x)=x_3,
 \end{cases} \]
where $ x=(x_0,x_1,x_2,x_3)\in \mathbb R^4 $.
It is evident  that $ F $ is an $ \mathbb O-$stem function since  (\ref{eq:variant-1}) holds true.
Then it induces a slice function $ f: \mathbb O \to \mathbb O $ given by
\[f(x_0+Ix_0+Jx_2+Kx_3)=3x_0+Ix_1+Jx_2+Kx_3,\]
for any $ \mathbb I=(1,I,J,K)\in \mathcal N $.
It is easy to verify  $ f $ is a solution of equations (\ref{DiracI}) by direct calculation, which means that $ f $ is a slice Dirac-regular function on $ \mathbb O $.
\end{example}
}

{
\begin{example}
We further generalize the above example, by constructing a function $ F=(F_0,F_1,F_2,F_3):\mathbb R^4 \to \mathbb O^4 $ such that
\[
\begin{cases}  F_0(x)=S(x_0,r),
\\
 F_1(x)=x_1h(x_0,r),
 \\
F_2(x)=x_2h(x_0,r),
\\
F_3(x)=x_3h(x_0,r),
 \end{cases} \]
where $ x=(x_0,x_1,x_2,x_3)\in \mathbb R^4 $ and $ r=\sqrt{x_1^2+x_2^2+x_3^2}$.
Here  $ S, h:\mathbb R^2 \to \mathbb O $  satisfy the differential equations
\begin{equation}\label{diff:equ}
\begin{cases}  y\partial_yh(x,y)+3h(x,y)=\partial_xS(x,y),
\\
y\partial_xh(x,y)=-\partial_yS(x,y).
 \end{cases}
 \end{equation}
It is direct to verify that (\ref{eq:variant-1}) holds so that  $ F $ is an $ \mathbb O-$stem function.
This stem function $F$  induces a slice function $ f: \mathbb O \to \mathbb O $ defined by
\[f(x_0+Ix_0+Jx_2+Kx_3)=S(x_0,r)+Ix_1h(x_0,r)+Jx_2h(x_0,r)+Kx_3h(x_0,r)\]
for any $ \mathbb I=(1,I,J,K)\in \mathcal N $.
Since  $ S,h $ satisfy equations (\ref{diff:equ}), it is easy to verify that  $ F $ is a solution of   equations (\ref{CFcondition}).
This  means that  $ f $ is a slice Dirac-regular function on $ \mathbb O $.
\\
\indent An explicit example for $F$ is given by
\[
\begin{cases}  F_0(x)=\frac{3}{5}x_0^5-2x_0^3r^2+\frac{3}{5}x_0r^4,
\\
 F_1(x)=x_1(x_0^4-\frac{6}{5}x_0^2r^2+\frac{3}{35}r^4),
 \\
F_2(x)=x_2(x_0^4-\frac{6}{5}x_0^2r^2+\frac{3}{35}r^4),
\\
F_3(x)=x_3(x_0^4-\frac{6}{5}x_0^2r^2+\frac{3}{35}r^4),
 \end{cases} \]
 but many others can be easily written.
\end{example}
}

The restriction of  a slice Dirac-regular function to a quaternionic space $\mathbb H_{\mathbb I} $ satisfy the following splitting property:

\begin{lemma}\label{splitting lemma}
Let $ f $ be a slice Dirac-regular function defined on an axially  symmetric domain $\Omega$.
Then for any $ \mathbb I \in \mathcal N $ and any $ e_4 \in \mathbb S^6 $ with $ e_4 \perp \mathbb H_{\mathbb I} $,
there exist two functions $ G_1,\ G_2: \Omega_{\mathbb I} \to \mathbb H_{\mathbb I} $ with $ D_{\mathbb I} G_1=0, \ \overline D_{\mathbb I} G_2=0 $
such that
\begin{equation}
f(q)=G_1(q) +e_4G_2(q),\ \ \ \forall\  q \in \Omega_{\mathbb I}.
\end{equation}
\end{lemma}

\begin{proof}
Since $ f $ is an octonion-valued function,   there exists an element $ e_4 \in \mathbb S^6 $ with $ e_4 \perp \mathbb H_{\mathbb I} $ such that
\begin{equation}\label{aaa}
f=G_1+e_4G_2,
\end{equation}
where $ G_1,\ G_2 \in \mathbb H_{\mathbb I} $.
Hence, also using Lemma \ref{ee}, it follows that
\begin{eqnarray*}
0=D_{\mathbb I} f&=&D_{\mathbb I} G_1+D_{\mathbb I} (e_4G_2)
\\
&=&D_{\mathbb I} G_1+e_4(\overline D_{\mathbb I} G_2),
\end{eqnarray*}
which implies  $$ D_{\mathbb I} G_1=\overline D_{\mathbb I} G_2=0 ,$$
and the assertion follows.
\end{proof}
\begin{remark}
We note that, in principle, one could have written $$f=G_1+G_2e_4$$ and the condition
of being slice Dirac regular would translate into $$D_{\mathbb I} G_1=0, \qquad
G_2 D_{\mathbb I} =0,$$  obtaining that $G_2$ is right regular.
\end{remark}

\section{Cauchy integral formula}
In this section, we present  the Cauchy integral theory for slice Dirac operator.

Throughout, we let $[\Omega]$ be an open subset in $\mathbb R^4$ and recall the notations
$$ \Omega:=\Big\{q=\mathbb I x^T \in \mathbb O : \mathbb I  \in \mathcal N, x  \in [\Omega] \Big\}, \qquad \Omega_{\mathbb I}=\Omega\cap\mathbb H_{\mathbb I}. $$
We shall consider the function $f:\Omega\longrightarrow \mathbb O$ and its restrictions $f_{\mathbb I }:=\left.f\right|_{\mathbb H_{\mathbb I}}$.

We let
  $$ {\bf n}(\xi)=n_0+In_1+Jn_2+Kn_3 $$  denote  the unit exterior normal to the boundary $ \partial \Omega_{\mathbb I} $ at $ \xi $.
We consider the Cauchy kernel in $\mathbb H_{\mathbb I}$ defined by
\begin{equation}\label{dues} V(\xi-q)=\frac{1}{2\pi^2}\frac{\overline{\xi-q}}{|\xi-q|^4},\qquad \forall\ \xi, q\in \mathbb H_{\mathbb I},
\end{equation}
and we finally let $$ dm=dx_0 \land dx_1 \land dx_2 \land dx_3 $$ be the Lebesgue volume element in $ \mathbb R^4 $, and $ dS $  the induced   surface element.

\begin{theorem}\label{Pompeiu formula}
Let $f:\Omega\longrightarrow \mathbb O$  be a slice   function on  a bounded axially symmetric set $ \Omega\subset\mathbb O$. Suppose that  $ f_{\mathbb I } \in C^1( \overline \Omega_{\mathbb I}) $ and $ \partial \Omega_{\mathbb I} $ is   piecewise smooth for some given  $ \mathbb I \in \mathcal N $.
Then for all $ q \in \Omega_{\mathbb I} $, we have
\begin{equation}
f_{\mathbb I}(q)=\int_{\partial \Omega_{\mathbb I}}V(\xi-q)({\bf n}(\xi) f_{\mathbb I}(\xi)) dS(\xi)-\int_{\Omega_{\mathbb I}} V(\xi-q)(D_{\mathbb I} f_{\mathbb I}(\xi))dm(\xi).
\end{equation}
\end{theorem}

\begin{proof} The classical divergence theorem shows that
$$ \int_{\Omega_{\mathbb I}} \frac{ \partial }{ \partial x_j} \mu dm= \int_{ \partial \Omega_{\mathbb I} }n_j \mu dS, \qquad  j=0,1,2,3. $$
for any real-valued function $ \mu \in C^1(\Omega_{\mathbb I}) \cap C(\overline \Omega_{\mathbb I}) $.
Thus for the octonion-valued function $ f_{\mathbb I} \in C^1(\Omega) \cap C(\overline \Omega) $, we have
 \begin{equation}\label{asterisco}
 \int_{\Omega_{\mathbb I}} \frac{ \partial }{ \partial x_j} f_{\mathbb I} dm= \int_{ \partial \Omega_{\mathbb I} } n_j f_{\mathbb I} dS,  \qquad j=0,1,2,3. \end{equation}
By Lemma \ref{splitting lemma}, there exist $ e_4 \in \mathbb S^6 $ with $ e_4 \perp \mathbb H_{\mathbb I} $ such that $ f_{\mathbb I}=G_1+e_4G_2 $ with $ G_1,\ G_2$ which are $\mathbb H_{\mathbb I} $-valued.  Hence for any map $ V:\Omega_{\mathbb I} \to \mathbb H_{\mathbb I}$ such that $V \in C^1(\Omega_{\mathbb I}) \cap C(\overline \Omega_{\mathbb I}) $, we have
\begin{equation}\label{G1}
\begin{aligned}
\int_{\Omega_{\mathbb I}}& (V D_{\mathbb I}) G_1 +V( D_{\mathbb I} G_1) dm
\\
=&\int_{\Omega_{\mathbb I}} \frac{\partial}{\partial x_1}( V G_1)+\frac{\partial}{\partial x_2}(V I G_1)+\frac{\partial}{\partial x_3}(V J G_1)+\frac{\partial}{\partial x_3}(V K G_1)dm
\\
=&\int_{\partial \Omega_{\mathbb I}} V {\bf n} G_1 dS,
\end{aligned}
\end{equation}
where we have used associativity in $\mathbb H_{\mathbb I}$.
Similarly, we have
\begin{equation}\label{G2}
\int_{\Omega_{\mathbb I}} (\overline{V} \ \overline {D_{\mathbb I}}) G_2 +\overline{V} ( \overline D_{\mathbb I} G_2) dm =  \int_{\partial \Omega_{\mathbb I}} \overline{V} \overline {\bf n} G_2 dS .
\end{equation}
The equalities \eqref{G1} and \eqref{G2} hold, in particular, when $V$ is the Cauchy kernel in \eqref{dues}.
We now fix  $ q  \in \Omega_{\mathbb I} $ and note that   $$ V(\xi-q)=-\frac{1}{4\pi^2}\overline D_\xi \frac{1}{|\xi-q|^2},$$
where $D_{\xi}$ denotes the Dirac operator with respect to the variable $\xi$.
Indeed,
\begin{equation}\label{V}
\begin{aligned}
V(\xi-q)=&\frac{-1}{4\pi^2}(\frac{\partial}{\partial x_0}-I\frac{\partial}{\partial x_1}-J\frac{\partial}{\partial x_2}-K\frac{\partial}{\partial x_3})\frac{1}{|\xi-q|^2}
\\
 =&\frac{1}{2\pi^2}\frac{\overline{\xi-q}}{|\xi-q|^4}.
\end{aligned}
 \end{equation}
Straightforward calculations show that
$$ (V D_\xi)(q)=(D_\xi V)(q)=0$$ for any $\xi\neq q$.

Take a sufficient small $ \varepsilon $ such that the ball $ B_\varepsilon(q)$ centered at $q$ and with radius $\varepsilon$ is contained in $\Omega_{\mathbb I} $. From  (\ref{G1}),  (\ref{G2}), and Lemma \ref{splitting lemma}
we have
\begin{equation}\label{f_{I,J}}
\int_{\Omega_{\mathbb I} \setminus B_\varepsilon (q)} V( D_\xi f_{\mathbb I}) dm =  \int_{\partial (\Omega_{\mathbb I}\setminus B_\varepsilon (q))} V ({\bf n } f_{\mathbb I}) dS .
\end{equation}
Hence we can calculate this integral as follows:
\begin{equation}\label{bbb}
\begin{aligned}
\int_{\Omega_{\mathbb I} \setminus B_\varepsilon(q)}V(\xi-q)(D_\xi f_{\mathbb I})dm=&\int_{\partial \Omega_{\mathbb I}}V(\xi-q)({\bf n } f_{\mathbb I})dS-\int_{\{|\xi-q|= \varepsilon\}} V(\xi-q)({\bf n} f_{\mathbb I})dS
\\
:=&I_{\partial \Omega_{\mathbb I}}-I_\varepsilon
\end{aligned}
\end{equation}
By Equation (\ref{V}) and Artin Theorem \ref{Artin}, we can evaluate  the limit of $ I_\varepsilon $:
\begin{equation}\label{1}
\begin{aligned}
\lim_{\varepsilon \to 0 }I_\varepsilon=&
  \frac{1}{2\pi^2}\lim_{\varepsilon \to 0}\int_{\{|\xi-q|=\varepsilon\}}\frac{\overline{\xi-q}}{|\xi-q|^4} \frac{\xi-q}{|\xi-q|}f_{\mathbb I}(\xi)dS
\\
=&\frac{1}{2\pi^2}\lim_{\varepsilon \to 0}\int_{\{|\xi-q|=\varepsilon\}}\frac{1}{|\xi-q|^3}f_{\mathbb I}(\xi)dS
\\
=&f_{\mathbb I}(q).
\end{aligned}
\end{equation}
Let $ \varepsilon \to 0 $ in (\ref{bbb}), we get
$$ f_{\mathbb I}(q)=\int_{\partial \Omega_{\mathbb I}}V(\xi-q)({\bf n}(\xi) f_{\mathbb I}(\xi)) dS-\int_{\Omega_{\mathbb I}} V(\xi-q)(D_\xi f_{\mathbb I}(\xi))dm,$$
as desired.
\end{proof}

\begin{corollary}\label{Cauchy formula} Let $ f $ be a slice Dirac-regular function on  a bounded axially symmetric set $\Omega $. Suppose that  $ f_{\mathbb I } \in C^1( \overline \Omega_{\mathbb I}) $ and $ \partial \Omega_{\mathbb I} $ is piecewise smooth for some given  $ \mathbb I \in \mathcal N $.
 then
\begin{equation}\label{eqC1}
f_{\mathbb I}(q)=\int_{ \partial \Omega_{\mathbb I} } V(\xi-q)({\bf n }(\xi)f_{\mathbb I}(\xi))dS,\ \ \ \forall\ q \in \Omega_{\mathbb I}
\end{equation}
and
\begin{equation}\label{eqC2}
\int_{ \partial \Omega_{\mathbb I} } V(\xi-q)({\bf n} (\xi)f_{\mathbb I}(\xi))dS=0, \ \ \ \forall\  q \notin \overline{\Omega}_{\mathbb I}.
\end{equation}
\end{corollary}

\begin{proof}
If  $ q \in \Omega_{\mathbb I} $,   then (\ref{eqC1}) follows from  Theorem \ref{Pompeiu formula}  since $ D_{\xi} f_{\mathbb I}(\xi)=0 $.
For $ q \notin \overline{\Omega}_{\mathbb I} $, the integral at the left hand side of (\ref{eqC2}) is a proper integral so that after limit process, (\ref{f_{I,J}}) becomes
$$ \int_{\Omega_{\mathbb I}}V(\xi-q)(D_\xi f_{\mathbb I})dm=\int_{\partial \Omega_{\mathbb I}}V(\xi-q)({\bf n } f_{\mathbb I})dS  $$
Since $f $ is slice Dirac-regular in $ \Omega $, the left hand side vanishes and we obtain  (\ref{eqC2}).
\end{proof}

Using the representation formula, we can introduce another kernel which allows to write a Cauchy formula of more general validity.

Denote $ M_{n\times m}(\mathbb O) $  the set of octonion matrix of $m$ rows and $n$ columns where $n, m$ are positive integer.
Given an octonion matrix $ A \in M_{n\times m}(\mathbb O) $,
the left multiplication operator $ L_A: M_{m\times k}(\mathbb O) \to  M_{n\times k}(\mathbb O)  $ defined as:
$$ L_AB:=AB,\ \ \ \forall B\in  M_{m\times k}(\mathbb O)$$
In general $ L_AL_B\neq L_{AB} $.

\begin{definition}
 Let $  \mathbb I' \in \mathcal N $ and $ \xi,q \in \mathbb H_{\mathbb I }$ be given.  We define the slice Cauchy kernel $\mathcal V(\xi, q,\mathbb I'): \mathbb O\to \mathbb O $
via
\begin{equation}\label{enne} \mathcal V(\xi, q,\mathbb I')=L_{\mathbb I'}L_{M_{\mathbb I}}L_{\mathbb V(\xi-q)},\end{equation}
where
$$ \mathbb V(\xi-q)=(V(\xi-q),V(\xi-\alpha(q)),V(\xi-\beta(q)),V(\xi-\gamma(q)))^T.$$
\end{definition}

\begin{theorem}\label{slice Cauchy} Let $ f $ be a slice function on  a bounded axially symmetric set $\Omega$,
suppose that  $ f_{\mathbb I } \in C^1( \overline \Omega_{\mathbb I}) $ and $ \partial \Omega_{\mathbb I} $ is piecewise smooth for some given  $ \mathbb I \in \mathcal N $.
Then for any $p\in\Omega$, there exists $ \mathbb I' \in \mathcal N $ such that $p=\mathbb I^\prime x^T$ with $x\in\mathbb R^4$ and
\begin{equation}
f(p)=\int_{ \partial \Omega_{\mathbb I} }\mathcal V(\xi, q,\mathbb I')({\bf n}(\xi)f_{\mathbb I}(\xi))dS-\int_{ \Omega_{
\mathbb I} }\mathcal V(\xi, q,\mathbb I')(D_{\mathbb I}f_{\mathbb I}(\xi))dm,
\end{equation}
where $ q=\mathbb I x^T \in \Omega_{\mathbb I}$ and $\mathcal V$ is as in \eqref{enne}.
\end{theorem}
\begin{proof}
By the  representation formula in Theorem \ref{rep.s formula}, for any $ q=\mathbb I x^T  \in \Omega_{\mathbb I} $ and   any $ \mathbb I' \in \mathcal N $ we have
\begin{equation}\label{M}
f(p)=f(\mathbb I'x^T)=\mathbb I'(M_{\mathbb I}\mathcal  F(q)).
\end{equation}
Since $ \Omega $ is an axially  symmetric set, it follows that $ \alpha(q),\beta(q),\gamma(q)\in \Omega_{\mathbb I} $ for any $ q \in \Omega_{\mathbb I}$. Theorem \ref{Pompeiu formula} gives
\begin{equation}\label{MM}
\mathcal F(q)=\int_{\partial \Omega_{\mathbb I}} \mathbb V(\xi-q)({\bf n}(\xi)f_{\mathbb I}(\xi))dS-\int_{\Omega_{\mathbb I}} \mathbb V(\xi-q)(D_{\mathbb I}f_{\mathbb I}(\xi))dm.
\end{equation}
Substituting (\ref{MM}) into (\ref{M}) and moving  the integral out, we finally get
\begin{equation*}
f(\mathbb I'x^T)=\int_{\partial \Omega_{\mathbb I}}\mathbb I'(M_{\mathbb I}(\mathbb V(\xi-q)({\bf n}(\xi)f_{\mathbb I}(\xi))))dS-\int_{\Omega_{\mathbb I}} \mathbb I'(M_{\mathbb I}(\mathbb V(\xi-q)(D_{\mathbb I}f_{\mathbb I}(\xi))))dm,
\end{equation*}
and \eqref{enne} allows to conclude.
\end{proof}

\section{Slice Dirac-regular power series}
In this section, we provide  the Taylor series for the slice Dirac-regular function  and the Laurent series for the slice Dirac-regular function near an isolated singularity.

For any  $ \alpha=(\alpha_1, \alpha_2, \alpha_3) \in \mathbb N^3 $, we set
$ n=|\alpha|=\alpha_1+\alpha_2+\alpha_3 $,
$ \alpha!=\alpha_1!\alpha_2!\alpha_3! $,
 $$ \partial_\alpha=\frac{\partial^n}{\partial x_1^{\alpha_1}\partial x_2^{\alpha_2}\partial x_3^{\alpha_3}}, $$
and
{for any $ q=x_0+ix_1+jx_2+kx_3\in \mathbb H $, denote}
$$ V_\alpha(q)=\partial_\alpha V(q), $$
where $$ V(q)=\frac{1}{2\pi^2}\frac{\overline q}{|q|^4}.$$
 It is well known that $ V_\alpha(q) $ is left and right Dirac-regular except at zero since $ \partial_\alpha $   commutes  with the Dirac operator $ D $.

Note that the monomials  $ f(q)=q^n$ are not Dirac-regular. Their Dirac-regular counterparts are
the homogeneous left and right Dirac-regular polynomials $ P_\alpha $, defined by
$$  P_\alpha(q)=\frac{\alpha !}{n!}\sum(x_{\beta_1}-i_{\beta_1}x_0)\cdots(x_{\beta_n}-i_{\beta_n}x_0), $$
where $ q=x_0+ix_1+jx_2+kx_3 $.
Here the sum runs over all $ \frac{n!}{a!} $ different orderings of $ \alpha_1 \ 1's $, $ \alpha_2 \ 2's $ and $ \alpha_3 \ 3's $
and $ i_{\beta_l} \in \{ i , j, k \}$ for any $l=1,2,\cdots,n$.

The polynomials  $ P_\alpha $ are homogeneous of degree $ n $, while $ V_\alpha $ is homogeneous of degree $ -n-3 $ (see \cite{Brackx_1}).

\noindent  Let $ U_n $ be  the right quaternionic vector space of homogeneous Dirac-regular functions of degree $ n\in\mathbb N $.
Then, the  polynomials $ P_\alpha \ (\alpha \in \mathbb N^3) $ are Dirac-regular and form a basis for $ U_n $.

\begin{theorem}\label{0}
 Let $ f $ be a slice Dirac-regular function in the unit ball $ B \subset \mathbb O $ centered at the origin and let $ f \in C^1(\overline{B}) $. For any $ q \in B $, there exist $ \mathbb I \in \mathcal N $ such that
 $ q \in \mathbb H_{\mathbb I} $, and
  \begin{equation}\label{tt}
 f(q)=\sum_{n=0}^{+\infty} \sum\limits_{\stackrel{\alpha \in \mathbb N^3}{|\alpha|=n}} P_{\alpha}(q)\frac{\partial_\alpha f_{\mathbb I}(0)}{\alpha!},
 \end{equation}
where the power series is uniformly convergent over $ B_{\mathbb I}$.
 \end{theorem}

 \begin{proof}
 Let $q\in B$, then there exists  $ \mathbb I \in \mathcal N $ such that
 $ q \in \mathbb H_{\mathbb I} $. Moreover, there exists a closed ball $ B_{\rho} $ with $ \rho < 1 $ such that $ q \in B_{\rho} $.
By Lemma \ref{splitting lemma}, we can pick  $e_4\in \mathbb S^6 $ with $ e_4 \perp \mathbb H_{\mathbb I} $, and write
$$ f_{\mathbb I}(q)=G_1(q)+e_4G_2(q),$$ where
$ G_1 $ and $ G_2 $ are $\mathbb H_{\mathbb I}-$valued Dirac-regular and conjugate Dirac-regular,  respectively. The integral formula (\ref{eqC1}) gives
 \begin{equation}\label{3}
 \begin{aligned}
  f_{\mathbb I}(q)=&\int_{\partial B_{\mathbb I}}V(\xi-q)({\bf n}(\xi)f_{\mathbb I}(\xi))dS
  \\
   =&\int_{\partial B_{\mathbb I}}V(\xi-q){\bf n}(\xi)G_1(\xi)dS+e_4\Big(\int_{\partial B_{\mathbb I}}\overline{V(\xi-q)} \ \overline{{\bf n}(\xi)}G_2(\xi)dS \Big)
   \\
   :=&I_1+e_4I_2
   \end{aligned}
   \end{equation}
 By Theorem 28 in \cite{Sudbery_1}, we can expand  $ V(\xi-q) $ in power series for any $ |q| < |\xi| $
 \begin{equation}\label{VV}
 \begin{aligned}
V(\xi-q)=&\sum_{n=0}^{+\infty}\sum\limits_{\stackrel{\alpha \in \mathbb N^3}{|\alpha|=n}}(-1)^nP_\alpha(q)V_\alpha(\xi)
 \\
 =&\sum_{n=0}^{+\infty}\sum\limits_{\stackrel{\alpha \in \mathbb N^3}{|\alpha|=n}}(-1)^nV_\alpha(\xi)P_\alpha(q)
\end{aligned}
\end{equation}
and the right hand side converges uniformly
 in any region $ \{ (\xi, q): |q| \le r|\xi| \} $ with $ r < 1 $.
Since $ q \in B_{\rho} $ and $ \xi \in \partial B $, we have $ |q | \le r|\xi| $ with $ r<1$. Using  the rightmost expression in (\ref{VV}), we get
\begin{equation}\label{1}
\begin{aligned}
 I_2=&\int_{\partial B_{\mathbb I}}\sum_{n=0}^{+\infty}\sum\limits_{\stackrel{\alpha \in \mathbb N^3}{|\alpha|=n}}(-1)^n\overline{V_\alpha(\xi)P_\alpha(q)}\ \overline {{\bf n}(\xi)} G_2(\xi)dS
  \\
 =&\sum_{n=0}^{+\infty}\sum\limits_{\stackrel{\alpha \in \mathbb N^3}{|\alpha|=n}}\overline {P_\alpha(q)}(-1)^n\int_{\partial B_{\mathbb I}} \overline{V_\alpha(\xi)}\ \overline{{\bf n}(\xi)} G_2(\xi)dS.
 \end{aligned}
 \end{equation}
Using the first expression in (\ref{VV}) and repeating  the procedure, we have
\begin{equation}\label{2}
I_1=\sum_{n=0}^{+\infty}\sum\limits_{\stackrel{\alpha \in \mathbb N^3}{|\alpha|=n}}P_\alpha(q)(-1)^n\int_{\partial B_{\mathbb I}} V_\alpha(\xi){\bf n}(\xi)G_2(\xi)dS.
\end{equation}
Substituting (\ref{1}) and (\ref{2}) into (\ref{3}) we obtain
$$ f_{\mathbb I}=\sum_{n=0}^{+\infty}\sum\limits_{\stackrel{\alpha \in \mathbb N^3}{|\alpha|=n}}P_\alpha(q)((-1)^n\int_{\partial B_{\mathbb I}} V_\alpha(\xi)({\bf n}(\xi)f_{\mathbb I}(\xi))dS). $$

\noindent Differentiating both sides of the integral formula (\ref{3}), we have
 $$ \partial_\alpha f_{\mathbb I}(q)=(-1)^n\int_{\partial B_{\mathbb I}}V_\alpha(\xi-q)({\bf n}(\xi)f_{\mathbb I}(\xi))dS. $$
In particular,  letting $ q \to 0 $ we conclude that
 $$   \partial_\alpha f_{\mathbb I}(0)=(-1)^n\int_{\partial B_{\mathbb I}} V_\alpha(\xi)({\bf n}(\xi)f_{\mathbb I}(\xi))dS. $$
 \end{proof}

{
\begin{remark} We point out that although the functions $ P_\alpha : \mathbb H_{\mathbb I} \to \mathbb O $  are homogeneous left and right Dirac-regular polynomials, they do not extend, in general, to a slice Dirac-regular function on the whole $ \mathbb O $.
For example, we consider the special case $ n=2 $ and set
\begin{equation*}
\begin{aligned}
f(q):=\sum\limits_{\stackrel{\alpha \in \mathbb N^3}{|\alpha|=2}}P_\alpha(q).
\end{aligned}
\end{equation*}
For any $q=x_0+Ix_1+Jx_2+Kx_3\in\mathbb H_{\mathbb I}$, by direct calculation we have
\begin{equation*}
\begin{aligned}
f(q)=F_0(x)+IF_1(x)+JF_2(x)+KF_3(x),
\end{aligned}
\end{equation*}
where
\begin{eqnarray*}
 F_0(x)&=&-6x_0^2+(x_1+x_2+x_3)^2+x_1^2+x_2^2+x_3^2
 \\
 F_1(x)&=&-2x_0(x_1+x_2+x_3)-2x_0x_1
  \\
 F_2(x)&=&-2x_0(x_1+x_2+x_3)-2x_0x_2
  \\
 F_3(x)&=&-2x_0(x_1+x_2+x_3)-2x_0x_3
 \end{eqnarray*}
Note that  $ F_0 $ does not satisfy the compatibility conditions (\ref{eq:variant-1}).
Therefore, not all $ P_\alpha $ can be  extended  to a slice Dirac-regular function on the whole $ \mathbb O $.
\end{remark}
}

\begin{remark} We still do not know if the series in (\ref{tt}) is convergent uniformly on the whole unit  ball $B$, besides on the subsets $B_{\mathbb I}$.
Our proof on $B_{\mathbb I}$ depends on the explicit  formula  of the kernel $V$ and the associativity of quaternions. This technique obviously fails in the setting of octonions and to consider the uniform  convergence  over $B$,  one should follow a different approach. In fact,
for any  $ f \in C^1(\overline{B}) $, one needs the estimate
 $$ |f(q)-f(q')|\le|f(\mathbb I x^T)-f(\mathbb I' x^T)|+| f(\mathbb I'x^T)-f(\mathbb I'x'^T)|,$$
and  it is problematic  to show  that  $ |f(\mathbb I x^T)-f(\mathbb I' x^T)| $ is small enough.
\end{remark}

We now come to study  the power series at any point $ q_0 \in \mathbb O $ for slice Dirac-regular functions.
With the same approach used in  Theorem \ref{0},
one  can show that
 \begin{equation}\label{q_0}
   f(q)=\sum_{n=0}^{+\infty}\sum\limits_{\stackrel{\alpha \in \mathbb N^3}{|\alpha|=n}}P_{\alpha}(q-q_0)\frac{\partial_\alpha f_{\mathbb I}(q_0)}{\alpha!}
 \end{equation}
for any $ q_0 \in \mathbb O \cap \mathbb H_{\mathbb I} $ and $q\in B_{\mathbb I}(q_0,R,\mathbb O):=B(q_0,R,\mathbb O)\cap \mathbb H_{\mathbb I}$.
Here $ B(q_0,R, \mathbb O)\subset \mathbb O $ denotes the ball of radius $ R $ centered at $ q_0 $.

Let  $ B(x_0,R)\subset \mathbb R^4 $ be the  ball of radius $ R $ centered at $ x_0 $ and denote by $\widetilde B(x_0, R)\subset \mathbb O $  the symmetrized set of  $ B(x_0, R) $.

We set
 $$ \mathbb P_{\alpha}(q-q_0)=(P_{\alpha}(q-q_0),\, P_{\alpha}(\alpha(q)-q_0),\, P_{\alpha}(\beta(q)-q_0),\,
 P_{\alpha}(\gamma(q)-q_0))^T$$
 and consider the operator
 $$ \mathcal P_\alpha(q ,q_0,\mathbb I')=L_{\mathbb I'}L_{M_{\mathbb I}}L_{\mathbb P_{\alpha}(q-q_0)}. $$

 \begin{theorem}  Assume that  $ f \in C^1(\overline{B(q_0,R, \mathbb O)})$ with $ R>0 $ and let $ q_0=\mathbb I x_0^T \in \mathbb H_{\mathbb I} $.
If  $ f $ is  slice Dirac-regular on  $ B(q_0,R, \mathbb O)$,
then
 $$ f( \mathbb I'x^T)=\sum_{n=0}^{+\infty}\sum\limits_{\stackrel{\alpha \in \mathbb N^3}{|\alpha|=n}}\mathcal P_{\alpha}(q, q_0, \mathbb I')\frac{\partial_{\alpha} f(q_0)}{\alpha!} $$
 for any $ q=\mathbb I x^T \in B_{\mathbb I}(q_0,R,\mathbb O)\cup \widetilde B(x_0, R)$.
 \end{theorem}

 \begin{proof}

By virtue of  (\ref{q_0}), $f$ admits the power series expansion
 $$ f(q)=\sum_{n=0}^{+\infty}\sum\limits_{\stackrel{\alpha \in \mathbb N^3}{|\alpha|=n}}P_{\alpha}(q-q_0)\frac{\partial_{\alpha} f(q_0)}{\alpha!}. $$
 for any  $q\in  B_{\mathbb I}(q_0,R,\mathbb O)$ and  $ q_0=\mathbb I x_0^T \in \mathbb H_{\mathbb I} $.
If  $ q \in \widetilde B(x_0, R) $, then  $ \alpha(q),\beta(q),\gamma(q) \in \widetilde B(x_0, R)$ so that
 $$ \mathcal F(q)=\sum_{n=0}^{+\infty}\sum\limits_{\stackrel{\alpha \in \mathbb N^3}{|\alpha|=n}}\mathbb P_{\alpha}(q-q_0)\frac{\partial_{\alpha} f(q_0)}{\alpha!}. $$
By the representation formula, we have
$$ f(\mathbb I' x^T)=\mathbb I'(M_{\mathbb I}\mathcal  F(q)). $$
Combining these two formulas and taking the sum we conclude that
$$ f(\mathbb I'x^T)=\sum_{n=0}^{+\infty}\sum\limits_{\stackrel{\alpha \in \mathbb N^3}{|\alpha|=n}}\mathcal P_{\alpha}(q, q_0, K)\frac{\partial_{\alpha} f(q_0)}{\alpha!}. $$
\end{proof}

Finally  we study  the Laurent power series.   We need to introduce some notation.
Let  $ q_0 \in \mathbb O $ and $0 \le R_1  <R_2 \le +\infty $. We consider the spherical shell in $\mathbb O$
$$ B(q_0,R_1,R_2, \mathbb O)=\{ q\in \mathbb O: \   R_1 <|q-q_0|<R_2  \}$$
and   the spherical shell in $\mathbb R^4$
$$ B(x_0,R_1,R_2)=\{ x_0\in \mathbb R^4: \   R_1 <|x-x_0|<R_2  \}.$$
 We let $\widetilde B(x_0,R_1, R_2) $ denote the symmetrized set of  $ B(x_0, R_1,R_2) $.

Let  $f \in C^1(\overline{B(q_0,R_1,R_2,\mathbb O)})$ and for $ q_0 \in \mathbb H_{\mathbb I }$, we set
$$S_i=\{ q \in \mathbb H_{\mathbb I}: |q-q_0|= R_i   \},\qquad i=1,2,$$
and the formulas:
$$ A_\alpha=(-1)^n\int_{S_2} V_\alpha(q-q_0)({\bf n}(\xi)f(\xi))dS, $$
$$ B_\alpha=(-1)^n\int_{S_1} P_\alpha(q-q_0)({\bf n}(\xi)f(\xi))dS.$$

\begin{theorem}\label{Laurent}
Let $ q_0=\mathbb I x_0^T \in  \mathbb H_{\mathbb I} $.
Let $ f $ be a slice Dirac-regular function on a spherical shell $ B(q_0,R_1,R_2) $ and $ f \in C^1(\overline{B(q_0,R_1,R_2,\mathbb O)})$.
Then
\[ f(\mathbb I'x^T)=\sum_{n=0}^{+\infty}\sum\limits_{\stackrel{\alpha \in \mathbb N^3}{|\alpha|=n}}[\mathcal P_\alpha(q, q_0, \mathbb I')A_\alpha+\mathcal V_\alpha(q, q_0, \mathbb I')B_\alpha] \]
for any $ q=\mathbb I x^T \in B_{\mathbb I}(q_0,R_1,R_2,\mathbb O)\cup \widetilde B(x_0,R_1,R_2)$.
\end{theorem}

\begin{proof}
Let $ q_0=\mathbb I x_0^T \in \mathbb H_{\mathbb I} $.
The integral formula in Theorem \ref{Pompeiu formula} implies
$$ f(q)=\int_{S_2}V(\xi-q)({\bf n}(\xi)f(\xi))dS-\int_{S_1}V(\xi-q)({\bf n}(\xi)f(\xi))dS. $$
For any $ \xi \in S_2 $, we have $ |\xi-q_0|>|q-q_0| $.
Therefore, the same approach as in the proof of  Theorem \ref{0} shows
$$  \int_{S_2}V(\xi-q)({\bf n}(\xi)f(\xi))dS=\sum_{n=0}^{+\infty}\sum\limits_{\stackrel{\alpha \in \mathbb N^3}{|\alpha|=n}}P_\alpha(q-q_0)A_\alpha. $$

For any  $ \xi \in S_1 $, we have $ |\xi-q_0|<|q-q_0| $. Now we use the second series in (\ref{VV}) and repeat the procedure in the proof Theorem \ref{0} to deduce that
\begin{equation}
\begin{aligned}
-\int_{S_1}&V(\xi-q)({\bf n}(\xi)f(\xi))dS
\\
=&\int_{S_1}V(q-\xi)({\bf n}(\xi)f(\xi))dS
\\
=&\int_{S_1}\sum_{n=0}^{+\infty}\sum\limits_{\stackrel{\alpha \in \mathbb N^3}{|\alpha|=n}}(-1)^nV_\alpha(q-q_0)P_\alpha(\xi-q_0)({\bf n}(\xi)f(\xi))dS
\\
=&\sum_{n=0}^{+\infty}\sum\limits_{\stackrel{\alpha \in \mathbb N^3}{|\alpha|=n}}V_\alpha(q-q_0)B_\alpha.
\end{aligned}
\end{equation}

\noindent This means
$$ f(q)=\sum_{n=0}^{+\infty}\sum\limits_{\stackrel{\alpha \in \mathbb N^3}{|\alpha|=n}}[P_\alpha(q-q_0)A_\alpha+V_\alpha(q-q_0)B_\alpha]. $$
For any $ q \in B_{\mathbb I}(q_0,R_1,R_2,\mathbb O) \cap \widetilde B(x_0,R_1,R_2) $, we have $$ \alpha(q),\beta(q),\gamma(q) \in B_{\mathbb I}(q_0,R_1,R_2,\mathbb O) \cap \widetilde B(x_0,R_1,R_2)$$  so that
$$ \mathcal F(q)=\sum_{n=0}^{+\infty}\sum\limits_{\stackrel{\alpha \in \mathbb N^3}{|\alpha|=n}}[\mathbb P_\alpha(q-q_0)A_\alpha+\mathbb V_\alpha(q-q_0)B_\alpha]. $$
By the representation formula
$$ f(\mathbb I'x^T)=\mathbb I'(M_{\mathbb I}\mathcal F_{\mathbb I}(q)). $$
we get the stated result
$$ f(\mathbb I'x^T)=\sum_{n=0}^{+\infty}\sum\limits_{\stackrel{\alpha \in \mathbb N^3}{|\alpha|=n}}[\mathcal P_\alpha(q, q_0, K)A_\alpha+\mathcal V_\alpha(q, q_0, K)B_\alpha]. $$
\end{proof}

\bibliographystyle{amsplain}

\end{document}